\documentclass[a4paper,english,sumlimits,reqno]{amsart}

\usepackage[T1]{fontenc}
\usepackage[utf8]{inputenc}
\usepackage{lmodern}

\usepackage{babel}

\usepackage{xifthen}%provides ifthenelse command
\usepackage{ifthenx}%complements xifthen. MUST be loaded after xifthen

\usepackage{xargs}

\usepackage{graphicx}
\usepackage[table,svgnames,x11names]{xcolor}

\usepackage[colorinlistoftodos,prependcaption,textsize=tiny]{todonotes}
\newcommandx{\todoin}[2][1=]{\todo[inline, caption={todo}, #1]{%
    \begin{minipage}{\textwidth-20pt}#2\end{minipage}}}
\newcommandx{\question}[2][1=]{\todo[linecolor=olive,backgroundcolor=olive!25,bordercolor=olive,#1]{#2}}
\newcommandx{\questionin}[2][1=]{\question[inline, caption={question}, #1]{%
    \begin{minipage}{\textwidth-40pt}{\normalsize #2}\end{minipage}}}

\usepackage{amssymb}
\usepackage{amsmath}
\usepackage{amsthm}
\usepackage{mathtools}
\usepackage{exscale}%to scale mathematical symbols
\usepackage{relsize}%to relatively adjust size (also works with math symbols)
\usepackage{bbm}%blackboard symbols
\usepackage{bm}
\usepackage{amsbsy}
\usepackage{mathdots}%provides \iddots

\usepackage[mathcal]{eucal}%provides \mathscr --> eulerscript
\usepackage{mathrsfs}%provides \mathscr --> eulerscript

\usepackage{stmaryrd}%provides double brackets \llbracket

\usepackage[all]{xy}
\newcommand{\vcxymatrix}[1]{\vcenter{\xymatrix{#1}}}

\usepackage{fp}

\usepackage[section]{placeins}
\usepackage{float}

\usepackage{enumerate}
\newcounter{proof}
\newenvironment{myproof}%
{\stepcounter{proof}\begin{proof}}%
{\end{proof}}%
\newcounter{proofstep}[proof]
\newenvironment{proofstep}[1][]%
{\refstepcounter{proofstep}\bigskip\par\noindent%
  \ifthenelse{\isempty{#1}}%if
    {\textsc{Step \theproofstep. }}%then
    {\textsc{#1.}}%else
  \noindent}%
{\par}%
\newcounter{proofcase}[proof]
{\refstepcounter{proofcase}\bigskip\par\noindent%
  \ifthenelse{\isempty{#1}}%if
    {\textsc{Case \theproofcase. }}%then
    {\textsc{#1.}}%else
  \noindent}%
{\par}%

\usepackage{hyperref}
\hypersetup{
  colorlinks,
  allcolors=DarkBlue
}
\theoremstyle{plain}
\newtheorem{thm}{Theorem}[section]
\newtheorem*{thm*}{Theorem}
\newtheorem{lem}[thm]{Lemma}

\theoremstyle{definition}

\theoremstyle{remark}

\newtheorem{rem}[thm]{Remark}

\numberwithin{equation}{section}

\newcommandx{\textref}[2][1=]{\hyperref[#2]{#1\ref*{#2}}}
\newcommandx{\textrefp}[2][1=]{(\hyperref[#2]{#1\ref*{#2}})}

\newcommand{\charfun}{\scalebox{1.0}{\ensuremath{\mathbbm 1}}}

\DeclareMathOperator*{\wslim}{w^*\negthinspace\lim}

\DeclareMathOperator{\Id}{Id}
\DeclareMathOperator{\spn}{span}

\begin{document}

\title{Subsymmetric weak$^*$ Schauder bases and factorization of the identity}

\author[R.~Lechner]{Richard Lechner}

\address{Richard Lechner, Institute of Analysis, Johannes
  Kepler University Linz, Altenberger Strasse 69, A-4040 Linz, Austria}
\email{Richard.Lechner@jku.at}

\date{\today}



\thanks{Supported by the Austrian Science Foundation (FWF) Pr.Nr. P28352}

\begin{abstract}
  Let $X^*$ denote a Banach space with a subsymmetric weak$^*$ Schauder basis satisfying
  condition~\eqref{eq:condition-c}.  We show that for any operator $T : X^*\to X^*$, either $T(X^*)$
  or $(\Id_{X^*}-T)(X^*)$ contains a subspace that is isomorphic to $X^*$ and complemented in $X^*$.
  Moreover, we prove that $\ell^p(X^*)$, $1\leq p \leq \infty$ is primary.
\end{abstract}

\maketitle

% \tableofcontents

% redefining a command that amsart redefined
\makeatletter
\providecommand\@dotsep{5}
\def\listtodoname{List of Todos}
\def\listoftodos{\@starttoc{tdo}\listtodoname}
\makeatother
% make list of todos
% \listoftodos%

\section{Introduction and results}
\label{sec:main-results}

\noindent
In~\cite{casazza:lin:1974}, Casazza and Lin showed that for any bounded linear projection $Q$ on a
Banach space $S$ with a subsymmetric basis, either $Q(S)$ or $(\Id-Q)(S)$ contains a subspace which
is isomorphic to $S$ and complemented in $S$.  Our first main result Theorem~\ref{thm:results:subsymmetric}
extends their result to Banach spaces having a subsymmetric weak$^*$~Schauder basis which satisfies
the subsequent condition~\eqref{eq:condition-c}.

\subsection*{Condition~\eqref{eq:condition-c}}
\label{sec:cond-eqref-c}
Let $U$ denote a Banach space with normalized \emph{unconditional} basis $(e_j)_{j=1}^\infty$, whose
normalized associated coordinate functionals $(e_j^*)_{j=1}^\infty$ form a \emph{unconditional}
weak$^*$~Schauder basis of $U^*$.  Given $\mathcal{A}\subset\mathbb{N}$, we define the bounded
projection $P_{\mathcal{A}} : U^*\to U^*$ by
\begin{equation}\label{eq:definition-P_A^*}
  P_{\mathcal{A}}\Big( \sum_{j=1}^\infty a_j e_j^* \Big)
  = \sum_{j\in \mathcal{A}}^\infty a_j e_j^*,
\end{equation}
where the above series converge in the weak$^*$ topology of $U^*$.

We say that $U^*$ together with its unconditional weak$^*$~Schauder basis $(e_j^*)_{j=1}^\infty$
satisfies \emph{condition~\eqref{eq:condition-c}}, if for every infinite set
$\Lambda\subset\mathbb{N}$ and every $\theta > 0$, we can find a sequence
$(\mathcal{A}_j)_{j=1}^\infty$ of pairwise disjoint and infinite subsets of $\Lambda$, such that for
all $(x_j^*)_{j=1}^\infty\subset U^*$ with $\|x_j^*\|_{U^*}\leq 1$ there exists a sequence of
scalars $(a_j)_{j=1}^\infty\in \ell^1$ with $\|(a_j)_{j=1}^\infty\|_{\ell^1} = 1$ such that
\begin{equation}\label{eq:condition-c}
  \Big\| \sum_{j=1}^\infty a_j P_{\mathcal{A}_j} x_j^* \Big\|_{U^*}
  \leq \theta.
  \tag{C}
\end{equation}
The above series converges in the weak$^*$ topology of $U^*$.

\begin{thm}\label{thm:results:subsymmetric}
  Let $S^*$ denote a Banach space with a subsymmetric weak$^*$~Schauder basis satisfying
  condition~\eqref{eq:condition-c}.  Then for any given bounded operator $T : S^*\to S^*$, there
  exist operators $M, N : S^*\to S^*$ such that for $H=T$ or $H = \Id_{S^*}-T$, the diagram
  \begin{equation}\label{eq:thm:results:subsymmetric}
    \vcxymatrix{S^* \ar[r]^{\Id_{S^*}} \ar[d]_{M} & S^*\\
      S^* \ar[r]_H & S^* \ar[u]_{N}}
    \qquad \|M\| \|N\| \leq 48 K_u^7 K_s^4
  \end{equation}
  is commutative.  Consequently, $H(S^*)$ contains a subspace which is isomorphic to $S^*$ and
  complemented in $S^*$.
\end{thm}
We prove Theorem~\ref{thm:results:subsymmetric} in Section~\ref{sec:subsymmetric}.

A related concept is the notion of a primary Banach space: a Banach space $X$ is \emph{primary}, if
for every bounded projection $Q : X\to X$ either $Q(X)$ or $(\Id - Q)(X)$ is isomorphic to $X$ (see
e.g.~\cite{lindenstrauss-tzafriri:1977}).  In~\cite{casazza:kottman:lin:1977}, Casazza, Kottman and
Lin showed that for any Banach space $S$ with a symmetric basis $(e_j)_{j=1}^\infty$, the following
Banach spaces are primary:
\begin{itemize}
\item $\bigl( \sum S \bigr) $, where the direct sum is either the $\ell^p$-sum for $1 < p < \infty$
  or the $c_0$-sum, and $S$ is not isomorphic to $\ell^1$;
\item $\bigl( \sum_n S_n \bigr)$, where the direct sum is either the $\ell^p$-sum for
  $1 < p < \infty$ or the $c_0$-sum, and $S_n = \spn\{e_1,\ldots,e_n\}$;
\item $\bigl( \sum \ell^\infty \bigr)_p$, $1\leq p < \infty$.
\end{itemize}
These results were later complemented by Capon in~\cite{capon:1982:1}, who showed that the Banach
spaces
\begin{itemize}
\item $\bigl( \sum S \bigr)_1$ and $\bigl( \sum S \bigr)_\infty$,
\item $\bigl( \sum_n S_n \bigr)_1$ and $\bigl( \sum_n S_n \bigr)_\infty$,
\end{itemize}
are primary.  Theorem~\ref{thm:results:primary} below adds $\ell^p(S^*)$ to that list ($S^*$ may be
non-separable).

\begin{thm}\label{thm:results:primary}
  Let $S^*$ denote a Banach space with a subsymmetric weak$^*$~Schauder basis satisfying
  condition~\eqref{eq:condition-c}.  Then for all $1\leq p \leq \infty$ and any given bounded
  operator $T : \ell^p(S^*)\to \ell^p(S^*)$, there exist operators
  $M,N : \ell^p(S^*)\to \ell^p(S^*)$ such that for $H=T$ or $H = \Id_{\ell^p(S^*)}-T$, the diagram
  \begin{equation}\label{eq:thm:results:primary}
    \vcxymatrix{\ell^p(S^*) \ar[rr]^{\Id_{\ell^p(S^*)}} \ar[d]_{M} && \ell^p(S^*)\\
      \ell^p(S^*) \ar[rr]_H && \ell^p(S^*) \ar[u]_{N}
    }
    \qquad \|M\| \|N\|\leq 48 K_u^7 K_s^4,
  \end{equation}
  is commutative.  Consequently, $\ell^p(S^*)$ is primary.
\end{thm}
The proof of Theorem~\ref{thm:results:primary} is given in Section~\ref{sec:sums-w-subsymmetric}.

The method of proof for both our main theorems is based on the recent
result~\cite{lechner:2016:factor-SL}.

\section{Notation}\label{sec:notation}

\noindent
Let $X$ be a Banach space and let $X^*$ denote its dual.  We say a sequence $(x_n^*)_{n=1}^\infty$
in $X^*$ \emph{converges weak$^*$}~to $x^*\in X^*$, if
\begin{equation*}
  \lim_{n} \langle x_n^*, x \rangle
  = \langle x^*, x\rangle,
  \qquad x\in X.  
\end{equation*}
In this case we write $\wslim x_n^* = x^*$.

We say that $(e_j^*)_{j=1}^\infty\subset X^*$ is a \emph{weak$^*$~Schauder basis} for $X^*$, if
there exists a basis $(e_j)_{j=1}^\infty$ for $X$ such that the \emph{associated coordinate
  functionals} of $(e_j){j=1}^\infty$ are $(e_j^*)_{j=1}^\infty$,
i.e.~$\langle e_j^*, e_j\rangle = 1$ and $\langle e_j^*, e_i\rangle = 1$, $i,j\in \mathbb{N}$,
$i\neq j$.  Thus,
\begin{equation*}
  x^* = \wslim_{n\to \infty} \sum_{j=1}^n \langle x^*, e_j\rangle e_j^*,
  \qquad x^*\in X^*.
\end{equation*}
For more details we refer to~\cite{singer:1970,lindenstrauss-tzafriri:1977}.  From now on, whenever
we encounter a series in the Banach space
\begin{itemize}
\item $X$, $U$ or $S$, the mode of convergence is in the norm topology of $X$, $U$ or $S$;
\item $X^*$, $U^*$ or $S^*$, the mode of convergence is in the weak$^*$~topology of $X^*$, $U^*$ or
  $S^*$.
\end{itemize}

We say the weak$^*$~Schauder basis $(e_j^*)_{j=1}^\infty$ of $X^*$ is \emph{unconditional}, if there
exists a constant $C \geq 1$ such that for all sequences of scalars $(a_j)_{j=1}^\infty$ with
$\sum_{j=1}^\infty a_j e_j^* \in X^*$ holds that
\begin{equation}\label{eq:unconditional}
  \Big\| \sum_{j=1}^\infty \gamma_j a_j e_j^* \Big\|_{X^*}
  \leq C \sup_{j\in \mathbb{N}} |\gamma_j| \Big\| \sum_{j=1}^\infty a_j e_j^* \Big\|_{X^*},
  \qquad
  (\gamma_j)_{j=1}^\infty\in \ell^\infty.
\end{equation}
We denote the infimum over all such constants $C\geq 1$ by $K_u$.  Note that
\begin{equation*}
  \text{$(e_j)_{j=1}^\infty$ is unconditional}
  \qquad\text{if and only if}\qquad
  \text{$(e_j^*)_{j=1}^\infty$ is unconditional}.
\end{equation*}

We say that a weak$^*$~Schauder basis $(e_j^*)_{j=1}^\infty$ of $X^*$ is \emph{subsymmetric}, if the
following two conditions are satisfied:
\begin{enumerate}[(a)]
\item $(e_j^*)_{j=1}^\infty$ is unconditional;
\item Let $(n_j)_{j=1}^\infty\subset \mathbb{N}$ be an increasing sequence, and let
  $Y_{(n_j)_{j=1}^\infty}$ denote the weak$^*$~closure of $\spn\{e_{n_j}^*\}_{j=1}^\infty$.  Then
  the operator $S_{(n_j)_{j=1}^\infty} : X^*\to Y_{(n_j)_{j=1}^\infty}$ given by
  \begin{equation}\label{eq:S:dfn}
    S_{(n_j)_{j=1}^\infty}\Big(\sum_{j=1}^\infty a_j e_j^* \Big)
    = \sum_{j=1}^\infty a_j e_{n_j}^*,
  \end{equation}
  is isomorphic.
\end{enumerate}
Note that
\begin{equation*}
  \text{$(e_j)_{j=1}^\infty$ is subsymmetric}
  \qquad\text{if and only if}\qquad
  \text{$(e_j^*)_{j=1}^\infty$ is subsymmetric}.
\end{equation*}
We define $K_s$ by
\begin{equation*}
  K_s
  = \sup \bigl\{ \|S_{(n_j)_{j=1}^\infty}\|,\|S_{(n_j)_{j=1}^\infty}^{-1}\| :
  \text{$(n_j)_{j=1}^\infty\subset \mathbb{N}$ is increasing}
  \bigr\}.
\end{equation*}
The constant $K_s$ is finite (see~\cite[Chapter~II, Theorem~21.2]{singer:1970}).

For all $1\leq p \leq \infty$ and for any Banach space $X$, we define as usual $\ell^p(X)$ by
\begin{equation*}
  \Bigl\{ (y_n)_{n=1}^\infty : y_n\in X\
  \ \text{and}\ 
  \|(y_n)_{n=1}^\infty\|_{\ell^p(X)}
  = \Bigl( \sum_{n=1}^\infty
  \|y_n\|_{X} \Bigr)^{1/p}
  < \infty \Bigr\}.
\end{equation*}

We conclude this section with two remarks on condition~\eqref{eq:condition-c}.
\begin{rem}\label{rem:condition-c:1}
  Any Banach space $U^*$ with an unconditional weak$^*$~Schauder basis $(e_j^*)$ in which
  \emph{$\ell^1$ does not embed}, satisfies condition~\eqref{eq:condition-c}.
\end{rem}

\begin{rem}\label{rem:condition-c:2}
  $\ell^\infty$ with the standard unit vector weak$^*$~Schauder basis $(e_j)_{j=1}^\infty$ satisfies
  condition~\eqref{eq:condition-c}.
\end{rem}

First, we will show that Remark~\ref{rem:condition-c:1} is indeed correct.  To this end, assume that
condition~\eqref{eq:condition-c} does not hold.  Then there exists an infinite set
$\Lambda\subset\mathbb{N}$ and a $\theta > 0$, such that for any sequence
$(\mathcal{A}_j)_{j=1}^\infty$ of pairwise disjoint and infinite subsets of $\Lambda$, we can find
$(x_j^*)_{j=1}^\infty\subset U^*$ with $\|x_j^*\|_{U^*}\leq 1$ such that for all scalars
$(a_j)_{j=1}^\infty\in \ell^1$ with $\|(a_j)_{j=1}^\infty\|_{\ell^1} = 1$ we have
\begin{equation*}
  \Big\| \sum_{j=1}^\infty a_j P_{\mathcal{A}_j} x_j^* \Big\|_{U^*}
  > \theta.
\end{equation*}
If we put $y_j^* = P_{\mathcal{A}_j} x_j^*$ and $\widetilde y_j^* = y_j^*/\|y_j^*\|_{U^*}$,
$j\in\mathbb{N}$, then we obtain by unconditionality
\begin{equation*}
  \Big\| \sum_{j=1}^\infty a_j \widetilde y_j^* \Big\|_{U^*}
  > \theta/K_u,
\end{equation*}
thus, $(\widetilde y_j^*)_{j=1}^\infty$ is equivalent to the standard unit vector basis of $\ell^1$.

To see that Remark~\ref{rem:condition-c:2} is valid, let $\Lambda\subset\mathbb{N}$ denote an infinite
set, and let $\theta > 0$.  Pick any sequence $(\mathcal{A}_j)_{j=1}^\infty$ of pairwise disjoint
subsets of $\Lambda$, and let $(x_j)_{j=1}^\infty$ denote any sequence of vectors in $\ell^\infty$
with $\|x_j\|_{\ell^\infty}\leq 1$, $j\in\mathbb{N}$.  Choose $N\in\mathbb{N}$ such that
$\frac{1}{N} \leq \theta$, and define $a_j = \frac{1}{N}$ if $j\leq N$, and $a_j = 0$ whenever
$j > N$.  Then clearly
\begin{equation*}
  \Bigl\| \sum_{j=1}^\infty a_j P_{\mathcal{A}_j} x_j \Bigr\|_{\ell^\infty}
  = \frac{1}{N}
  \leq \theta.
\end{equation*}

\section{Two subspace annihilation lemmata}\label{sec:annihilate}

\noindent
We will present two lemmata, which are used in the proof of our first main result
Theorem~\ref{thm:results:subsymmetric} to almost diagonalize an operator $T : X^*\to X^*$.  Although we use
Lemma~\ref{lem:past} only for $m=1$ here, we give the proof for general $m\in\mathbb{N}$ here, for later
reference.
\begin{lem}\label{lem:past}
  Let $X$ denote a Banach space with a normalized basis $(e_j)_{j=1}^\infty$, such that the
  associated coordinate functionals $(e_j^*)_{j=1}^\infty$ form a weak$^*$~Schauder basis of $X^*$.
  Let $\Lambda_0\subset\mathbb{N}$ denote an infinite set, $m,n\in \mathbb{N}$,
  $x_1^*,\ldots,x_n^*\in X^*$ and $\eta>0$.  Then we can find a finite set
  $\mathcal{F}\subset\Lambda_0$ with cardinality $|\mathcal{F}| = 2m$ such that
  \begin{equation*}
    \max_{1\leq j \leq n}
    \Big|\Big\langle \sum_{k\in \mathcal{F}} \varepsilon_k e_k, x_j^* \Big\rangle\Big|
    \leq \eta,
    \qquad \varepsilon\in \mathcal{E},
  \end{equation*}
  where the set of signs $\mathcal{E} = \mathcal{E}(\mathcal{F})$ is given by
  \begin{equation*}
    \mathcal{E}
    = \Big\{ (\varepsilon_k)_{k\in \mathcal{F}} :
    \varepsilon_k\in \{\pm 1\},\ \sum_{k\in \mathcal{F}} \varepsilon_k = 0 \Big\}.
  \end{equation*}
\end{lem}

\begin{proof}
  Firstly, we will prove that given $x^*\in X^*$, there exists an infinite set
  $\Lambda\subset \Lambda_0$ such that
  \begin{equation}\label{proof:lem:past:1}
    | \langle e_k, x^* \rangle - \langle e_{k'}, x^* \rangle |
    \leq \frac{\eta}{m},
    \qquad k,k'\in \Lambda.
  \end{equation}
  To this end, we define the sets
  \begin{equation*}
    \mathcal{M}_\ell = \Big\{ j\in \Lambda_0 :
    \frac{\eta}{2m} (\ell-1) \leq \langle e_j, x^* \rangle < \frac{\eta}{2m} \ell
    \Big\},
    \qquad \ell\in \mathbb{Z}.
  \end{equation*}
  Since $|\langle e_j, x^* \rangle| \leq \|x^*\| < \infty$, there are only finitely many
  $\ell\in \mathbb{Z}$ for which $\mathcal{M}_\ell$ is non-empty.  Hence, since
  $\bigcup_{\ell\in\mathbb{Z}} \mathcal{M}_l = \Lambda_0$, at least one of the sets
  $\mathcal{M}_\ell$, $\ell\in \mathbb{Z}$ has to be infinite.  Clearly, this infinite set
  satisfies~\eqref{proof:lem:past:1}.

  Now we will just repeatedly use~\eqref{proof:lem:past:1}.  We begin by selecting
  $\Lambda_1\subset\Lambda_0$ so that
  \begin{equation*}
    | \langle e_k, x_1^* \rangle - \langle e_{k'}, x_1^* \rangle |
    \leq \frac{\eta}{m},
    \qquad k,k'\in \Lambda_1.
  \end{equation*}
  Let $1 \leq i \leq n-1$ and assume we have already chosen infinite sets
  $\Lambda_0\supset\Lambda_1\supset \Lambda_2 \supset \cdots \supset \Lambda_{i-1}$ such that
  \begin{equation*}
    \max_{1\leq j \leq i-1} | \langle e_k, x_j^* \rangle - \langle e_{k'}, x_j^* \rangle |
    \leq \frac{\eta}{m},
    \qquad k,k'\in \Lambda_{i-1}.
  \end{equation*}
  By replacing $\Lambda_0$ with the infinite set $\Lambda_{i-1}$ in~\eqref{proof:lem:past:1}, we can
  find an infinite set $\Lambda_i\subset \Lambda_{i-1}$ such that
  \begin{equation*}
    \max_{1\leq j \leq i} | \langle e_k, x_j^* \rangle - \langle e_{k'}, x_j^* \rangle |
    \leq \frac{\eta}{m},
    \qquad k,k'\in \Lambda_i.
  \end{equation*}
  Stopping the induction after $n$ steps, we obtain
  \begin{equation}\label{proof:lem:past:2}
    \max_{1\leq j \leq n} | \langle e_k, x_j^* \rangle - \langle e_{k'}, x_j^* \rangle |
    \leq \frac{\eta}{m},
    \qquad k,k'\in \Lambda_n.
  \end{equation}

  We choose any $\mathcal{F}\subset \Lambda_n$ with $|\mathcal{F}| = 2m$ and let
  $\varepsilon = (\varepsilon_k)_{k\in \mathcal{F}}\in \mathcal{E} = \mathcal{E}(\mathcal{F})$,
  i.e.~$\varepsilon_k\in \{\pm 1\}$ is such that $\sum_{k\in \mathcal{F}} \varepsilon_k = 0$.  Then
  since $\mathcal{F}\subset \Lambda_n$, we obtain from~\eqref{proof:lem:past:2} that
  \begin{equation*}
    \Big|\Big\langle \sum_{k\in \mathcal{F}} \varepsilon_k e_k, x_j^* \Big\rangle\Big|
    = \Big|
    \sum_{\substack{k\in \mathcal{F}\\\varepsilon_k=1}} \langle e_k, x_j^*\rangle
    - \sum_{\substack{k\in \mathcal{F}\\\varepsilon_k=-1}} \langle e_k, x_j^*\rangle
    \Big|
    \leq \eta,
    \qquad 1\leq j \leq n.
    \qedhere
  \end{equation*}
\end{proof}

The following Lemma is an abstract version of an argument which Lindenstrauss used
in~\cite{lindenstrauss:1967} (see also~\cite{lindenstrauss-tzafriri:1977}) to show that
$\ell^\infty$ is prime (which means that every infinite dimensional complemented subspace of
$\ell^\infty$ is isomorphic to $\ell^\infty$).
\begin{lem}\label{lem:future}
  Let $U$ denote a Banach space with a normalized unconditional basis $(e_j)_{j=1}^\infty$, such
  that the associated coordinate functionals $(e_j^*)_{j=1}^\infty$ form an unconditional
  weak$^*$~Schauder basis of $U^*$.  Let $\Lambda\subset\mathbb{N}$ denote an infinite set,
  $\eta > 0$ and $\varphi\in U^{**}$.  If $U^*$ together with $(e_j^*)_{j=1}^\infty$ satisfies
  condition~\eqref{eq:condition-c}, then there exists an infinite set $\mathcal{A}\subset \Lambda$
  such that
  \begin{equation*}
    \sup_{\|x^*\|_{U^*}\leq 1} |\langle \varphi, P_{\mathcal{A}} x^* \rangle| \leq \eta.
  \end{equation*}
\end{lem}

\begin{proof}
  Let $\eta > 0$, $\varphi\in U^{**}\setminus\{0\}$, and assume the conclusion of the Lemma is
  false.  Hence, for all infinite sets $\mathcal{A}\subset\Lambda$ there exists $x^*$ with
  $\|x^*\|_{U^*} = 1$ such that
  \begin{equation}\label{proof:lem:future:0}
    \langle \varphi, P_{\mathcal{A}} x^* \rangle
    > \eta.
  \end{equation}
  We define $\theta = \frac{\eta}{2\|\varphi\|_{U^{**}}K_u}$ and choose
  $(\mathcal{A}_j)_{j=1}^\infty$ with $\mathcal{A}_j\subset\Lambda$, $j\in\mathbb{N}$ according to
  condition~\eqref{eq:condition-c}.  By our assumption~\eqref{proof:lem:future:0}, we can find
  $x^*_j\in {U^*}$ with $\|x^*_j\|_{U^*}=1$, $j\in \mathbb{N}$, such that
  \begin{equation}\label{proof:lem:future:1}
    \langle \varphi, P_{\mathcal{A}_j} x^*_j \rangle > \eta.
  \end{equation}
  By condition~\eqref{eq:condition-c} we can find a sequence $(a_j)_{j=1}^\infty\in \ell^1$ with
  $\|(a_j)_{j=1}^\infty\|_{\ell^1} = 1$ such that
  \begin{equation}\label{proof:lem:future:2}
    \Big\| \sum_{j=1}^\infty a_j P_{\mathcal{A}_j} x^*_j \Big\|_{U^*}
    \leq \theta
    = \frac{\eta}{2\|\varphi\|_{U^{**}}K_u}.
  \end{equation}
  But on the other hand, we obtain from~\eqref{proof:lem:future:1} that
  \begin{equation*}
    \Big\langle \varphi, \sum_{j=1}^n a_j P_{\mathcal{A}_j} x^*_j \Big\rangle
    > \eta,
  \end{equation*}
  for a large enough $n\in \mathbb{N}$.  Combining the latter estimate
  with~\eqref{proof:lem:future:2}, we arrive at the contradiction
  \begin{equation*}
    \eta
    < \Big\langle \varphi, \sum_{j=1}^n a_j P_{\mathcal{A}_j} x^*_j \Big\rangle
    \leq \|\varphi\|_{U^{**}} \Big\| \sum_{j=1}^n a_j P_{\mathcal{A}_j} x^*_j \Big\|_{U^*}
    \leq \eta/2.
    \qedhere
  \end{equation*}
\end{proof}

\section{Proof of Theorem~\ref{thm:results:subsymmetric}}\label{sec:subsymmetric}

\noindent
For convenience of the reader, we restate Theorem~\ref{thm:results:subsymmetric} here below.
\begin{thm}[Main result Theorem~\ref{thm:results:subsymmetric}]\label{thm:subsymmetric}
  Let $S^*$ denote a Banach space with a subsymmetric weak$^*$~Schauder basis satisfying
  condition~\eqref{eq:condition-c}.  Then for any given bounded operator $T : S^*\to S^*$, there
  exist operators $M, N : S^*\to S^*$ such that for $H=T$ or $H = \Id_{S^*}-T$, the diagram
  \begin{equation}\label{eq:thm:subsymmetric}
    \vcxymatrix{S^* \ar[r]^{\Id_{S^*}} \ar[d]_{M} & S^*\\
      S^* \ar[r]_H & S^* \ar[u]_{N}}
    \qquad \|M\| \|N\| \leq 48 K_u^7 K_s^4
  \end{equation}
  is commutative.  Consequently, $H(S^*)$ contains a subspace which is isomorphic to $S^*$ and
  complemented in $S^*$.
\end{thm}

\begin{myproof}
  In this proof, we use the following constants:
  \begin{equation}\label{eq:proof:thm:subsymmetric:eta}
    \eta_i = K_u^{-1}4^{-i-1},
    \qquad i\in \mathbb{N}.
  \end{equation}

  \begin{proofstep}[Step~\theproofstep: Inductive construction of the block basis]
    Let $(s_j)_{j=1}^\infty$ denote the subsymmetric basis of $S$, and let $(s_j^*)_{j=1}^\infty$
    denote the associated coordinate functionals, which form a weak$^*$ Schauder basis for $S^*$,
    such that condition~\eqref{eq:condition-c} is satisfied.  We put $\mathcal{A}_1 = \mathbb{N}$,
    choose $\mathcal{B}_1 = \{1,2\}$, and we define
    \begin{equation}\label{eq:proof:thm:subsymmetric:block-initial}
      b_1 = s_1 - s_2
      \qquad\text{and}\qquad
      b_1^* = s_1^* - s_2^*.
    \end{equation}
    By Lemma~\ref{lem:future}, there exists an infinite collection
    $\mathcal{A}_2\subset \mathcal{A}_1\setminus\{1,2\}$, such that
    \begin{equation}\label{eq:proof:thm:subsymmetric:block-initial:estimate}
      \sup_{\|x^*\|_{S^*}\leq 1} |\langle T^* b_1, P_{\mathcal{A}_2} x^* \rangle|
      \leq \eta_1.
    \end{equation}

    Now assume we have already chosen infinite sets
    $\mathcal{A}_1\supset \mathcal{A}_2\supset \cdots \supset \mathcal{A}_i$, pairwise disjoint
    finite sets $\mathcal{B}_j\subset \mathcal{A}_j$ with $|\mathcal{B}_j|=2$, $1\leq j \leq i-1$
    and that we have defined
    \begin{equation}\label{eq:proof:thm:subsymmetric:block-hypothesis}
      b_j = s_{k_0} - s_{k_1}
      \quad\text{and}\quad
      b_j^* = s_{k_0}^* - s_{k_1}^*,
      \qquad k_0,k_1\in \mathcal{B}_j, k_0 < k_1,
    \end{equation}
    for all $1\leq j \leq i-1$.  By Lemma~\ref{lem:past}, we can find a set
    $\mathcal{B}_i\subset \mathcal{A}_i$ with $|\mathcal{B}_i| = 2$ so that if we put
    \begin{equation}\label{eq:proof:thm:subsymmetric:block-step:1}
      b_i = s_{k_0} - s_{k_1}
      \quad\text{and}\quad
      b_i^* = s_{k_0}^* - s_{k_1}^*,
      \qquad k_0,k_1\in \mathcal{B}_i, k_0 < k_1,
    \end{equation}
    we have the estimate
    \begin{equation}\label{eq:proof:thm:subsymmetric:estimate-step:1}
      \sum_{j=1}^{i-1} |\langle b_i, T b_j^* \rangle|
      \leq \eta_i.
    \end{equation}
    By Lemma~\ref{lem:future}, there exists an infinite set $\mathcal{A}_{i+1}$ with
    \begin{equation}\label{eq:proof:thm:subsymmetric:A:dfn}
      \mathcal{A}_{i+1}\subset \mathcal{A}_i\setminus \Bigl\{ k\in \mathbb{N} : k\leq
      \max\big(\bigcup_{j=1}^i \mathcal{B}_i\big) \Bigr\}
    \end{equation}
    such that
    \begin{equation}\label{eq:proof:thm:subsymmetric:estimate-step:2}
      \sup_{\|x^*\|_{S^*}\leq 1} \Big|\Big\langle T^* b_i,
      P_{\mathcal{A}_{i+1}} x^* \Big\rangle\Big|
      \leq \eta_i.
    \end{equation}
    This completes the inductive step.

    The
    estimates~\eqref{eq:proof:thm:subsymmetric:block-initial:estimate},~\eqref{eq:proof:thm:subsymmetric:estimate-step:1}
    and~\eqref{eq:proof:thm:subsymmetric:estimate-step:2} imply
    \begin{subequations}\label{eq:proof:thm:subsymmetric:estimate:off-diag}
      \begin{align}
        \sum_{j=1}^{i-1} |\langle b_i, T b_j^* \rangle|
        & \leq \eta_i,
        && i\in \mathbb{N},
           \label{eq:proof:thm:subsymmetric:estimate:off-diag:a}\\
        \Big|\Big\langle b_i,
        T \sum_{j=i+1}^\infty a_j b_j^*\Big\rangle\Big|
        & \leq \eta_i \Big\|\sum_{j=i+1}^\infty a_j b_j^*\Big\|_{S^*},
        && i\in \mathbb{N},
           \label{eq:proof:thm:subsymmetric:estimate:off-diag:b}
      \end{align}
      whenever $(a_j)_{j=1}^\infty\subset \mathbb{R}$ is such that
      $\sum_{j=i+1}^\infty a_j b_j^* \in S^*$.
    \end{subequations}
  \end{proofstep}

  \begin{proofstep}[Step~\theproofstep: Basic operators]
    We define the operators $B,Q : S^*\to S^*$ by
    \begin{subequations}\label{eq:BQ}
      \begin{align}
        % B\Big( \sum_{j=1}^\infty a_j s_j^* \Big)
        % & = \sum_{j=1}^\infty \frac{1}{\|b_j^*\|_{S^*}} a_j b_j^*,
            B x^*
        & = \sum_{j=1}^\infty \frac{1}{\|b_j^*\|_{S^*}} \langle x^*, s_j\rangle b_j^*,
        && x^*\in S^*,
           \label{eq:BQ:a}\\
        Q x^*
          & = \sum_{j=1}^\infty
            \frac{\|b_j^*\|_{S^*}}{|\mathcal{B}_j|}
            \langle b_j, x^*\rangle s_j^*,
        && x^*\in S^*.
           \label{eq:BQ:b}
      \end{align}
    \end{subequations}

    We will now show that $B$ and $Q$ are bounded linear operators. To this end, let
    $x^* = \sum_{j=1}^\infty a_j s_j^*$.  Firstly, note that
    \begin{equation*}
      \Big\|B\Big( \sum_{j=1}^\infty a_j s_j^* \Big)\Big\|_{S^*}
      \leq \Big\|\sum_{j=1}^\infty \frac{1}{\|b_j^*\|_{S^*}}
      a_j s_{k_0(j)}^* \Big\|_{S^*}
      + \Big\|\sum_{j=1}^\infty \frac{1}{\|b_j^*\|_{S^*}}a_j s_{k_1(j)}^* \Big\|_{S^*},
    \end{equation*}
    where $\{k_0(j),k_1(j)\} = \mathcal{B}_j$, $j\in \mathbb{N}$.  By unconditionality, we obtain
    $\|b_j^*\|_{S^*} \geq K_u^{-1}$ and
    \begin{equation*}
      \Big\|B\Big( \sum_{j=1}^\infty a_j s_j^* \Big)\Big\|_{S^*}
      \leq K_u^2\Big(\Big\|\sum_{j=1}^\infty a_j s_{k_0(j)}^* \Big\|_{S^*}
      + \Big\|\sum_{j=1}^\infty a_j s_{k_1(j)}^* \Big\|_{S^*}\Big),
    \end{equation*}
    Furthermore, by~\eqref{eq:proof:thm:subsymmetric:A:dfn} we have that $s_{k_0(j)} < s_{k_0(j+1)}$
    and $s_{k_1(j)} < s_{k_1(j+1)}$, $j\in \mathbb{N}$.  The weak$^*$~Schauder basis
    $(s_j^*)_{j=1}^\infty$ is subsymmetric, hence
    \begin{equation}\label{eq:B:estimate}
      \Big\|B\Big( \sum_{j=1}^\infty a_j s_j^* \Big)\Big\|_{S^*}
      \leq 2 K_u^2K_s\Big\|\sum_{j=1}^\infty a_j s_j^* \Big\|_{S^*}.
    \end{equation}

    Secondly, if we write $b_j = s_{k_0}(j) - s_{k_1}(j)$ for
    $k_0,k_1\in \mathcal{B}_j$, $k_0<k_1$, $j\in \mathbb{N}$, we obtain
    \begin{equation*}
      \|Q x^*\|_{S^*}
      \leq \Big\|
      \sum_{j=1}^\infty \frac{\|b_j^*\|_{S^*}}{|\mathcal{B}_j|}
      \langle s_{k_0(j)}, x^*\rangle s_j^*
      \Big\|_{S^*}
      + \Big\|
      \sum_{j=1}^\infty \frac{\|b_j^*\|_{S^*}}{|\mathcal{B}_j|}
      \langle s_{k_1(j)}, x^*\rangle s_j^*
      \Big\|_{S^*}.
    \end{equation*}
    Since $(s_j^*)_{j=1}^\infty$ is subsymmetric, the right hand side of the latter inequality is
    dominated by
    \begin{equation*}
      K_s \Big\|
      \sum_{j=1}^\infty \frac{\|b_j^*\|_{S^*}}{|\mathcal{B}_j|}
      \langle s_{k_0(j)}, x^*\rangle s_{k_0(j)}^*
      \Big\|_{S^*}
      + K_s \Big\|
      \sum_{j=1}^\infty \frac{\|b_j^*\|_{S^*}}{|\mathcal{B}_j|}
      \langle s_{k_1(j)}, x^*\rangle s_{k_1(j)}^*
      \Big\|_{S^*}.
    \end{equation*}
    Since $\frac{\|b_j^*\|_{S^*}}{|\mathcal{B}_j|}\leq 1$, unconditionality yields
    \begin{equation}\label{eq:Q:estimate}
      \|Q x^*\|_{S^*}
      \leq 2 K_s K_u \|x^*\|_{S^*}.
    \end{equation}

    One can easily verify that $Q B = \Id_{S^*}$, i.e.~the diagram
    \begin{equation}\label{eq:BQ:commutative-diagram}
      \vcxymatrix{S^* \ar[rr]^{\Id_{S^*}} \ar[dr]_{B} & & S^*\\
        & S^* \ar[ru]_{Q} & }
    \end{equation}
    is commutative.  Consequently, $B$ is an isomorphism onto its range and its
    range is complemented by $BQ$.
  \end{proofstep}

  \begin{proofstep}[Step~\theproofstep: Conclusion of the proof]
    Observe that at least one of the two following sets is infinite:
    \begin{equation*}
      \{j\in \mathbb{N} : |\langle b_j, Tb_j^*\rangle| \geq 1\}
      \qquad\text{or}\qquad
      \{j\in \mathbb{N} :
      |\langle b_j, (\Id_{S^*}-T)b_j^*\rangle| \geq 1\}.
    \end{equation*}
    If the left set is infinite we denote it by $\mathcal{J}$ and put $H=T$, and if the left set is
    finite, we denote the right set by $\mathcal{J}$ and we define $H=\Id_{S^*}-T$.  In either case,
    we obtain
    \begin{equation}\label{eq:either-or-large}
      \mathcal{J}\ \text{is infinite}
      \qquad\text{and}\qquad
      |\langle b_j, Hb_j^*\rangle| \geq 1,
      \quad j\in\mathcal{J}.
    \end{equation}
    Let $Y = B(S^*)$ and note that~\eqref{eq:B:estimate},~\eqref{eq:Q:estimate}
    and~\eqref{eq:BQ:commutative-diagram} yields
    \begin{equation}\label{eq:Y:commutative-diagram}
      \vcxymatrix{
        S^* \ar[rr]^{\Id_{S^*}} \ar[d]_B & & S^*\\
        Y \ar[rr]_{\Id_{S^*}} & & Y \ar[u]_{Q_{|Y}}
      }
      \qquad \|B\| \|Q_{|Y}\| \leq 4 K_u^3 K_s^2.
    \end{equation}

    Now, define $P : S^*\to Y$ by
    \begin{equation}\label{eq:U}
      Px^*
      =  \sum_{j\in \mathcal{J}}
      \frac{\langle b_j, x^*\rangle}
      {\langle b_j, H b_j^*\rangle} b_j^*,
      \qquad x^*\in S^*,
    \end{equation}
    and observe that by~\eqref{eq:either-or-large}, the unconditionality of $(s_j^*)_{j=1}^\infty$
    (and thus, the unconditionality of $(b_j^*)_{j=1}^\infty$) and the definition of $B$ and $Q$
    (see~\eqref{eq:BQ}), we obtain
    \begin{equation*}
      \|Px^*\|_{S^*}
      \leq 2 K_u \|BQx^*\|_{S^*},
      \qquad x^*\in S^*.
    \end{equation*}
    Combining the latter estimate with~\eqref{eq:B:estimate} and~\eqref{eq:Q:estimate} yields
    \begin{equation}\label{eq:U:estimate}
      \|Px^*\|_{S^*}
      \leq 8 K_u^4 K_s^2 \|x^*\|_{S^*},
      \qquad x^*\in S^*.
    \end{equation}

    A straightforward calculation shows that for all $y = \sum_{j\in \mathcal{J}} a_j b_j^* \in Y$,
    the following identity is true:
    \begin{equation}\label{eq:crucial-identity}
      PHy - y
      = \sum_{i\in \mathcal{J}}
      \sum_{j\in \mathcal{J} : j < i}
      a_j
      \frac{\langle b_i, H b_j^*\rangle}
      {\langle b_i, Hb_i^*\rangle}
      b_i^*
      + \sum_{i\in \mathcal{J}}
      \frac{\big\langle b_i,
        H \sum_{j\in \mathcal{J} : j > i} a_j b_j^*\big\rangle}
      {\langle b_i, Hb_i^*\rangle}
      b_i^*.
    \end{equation}
    Since $|\langle b_j, y\rangle| \leq \|b_j\|_S \|y\|_{S^*}$, we obtain $|a_j| \leq \|y\|_{S^*}$.
    Hence, by~\eqref{eq:either-or-large} and the crucial off diagonal
    estimate~\eqref{eq:proof:thm:subsymmetric:estimate:off-diag:a}, we obtain
    \begin{equation}\label{eq:U:invert:estimate:1}
      \bigg\|\sum_{i\in \mathcal{J}}
      \sum_{j\in \mathcal{J} : j < i}
      a_j
      \frac{\langle b_i, H b_j^*\rangle}
      {\langle b_i, Hb_i^*\rangle}
      b_i^*
      \bigg\|_{S^*}
      \leq 2\|y\|_{S^*} \sum_{i\in \mathbb{N}} \eta_i.
    \end{equation}
    Using~\eqref{eq:either-or-large} and the second off diagonal
    estimate~\eqref{eq:proof:thm:subsymmetric:estimate:off-diag:b}, we obtain by unconditionality
    \begin{equation}\label{eq:U:invert:estimate:2}
      \begin{aligned}
        \bigg\|\sum_{i\in \mathcal{J}} \frac{\big\langle b_i, H \sum_{j\in \mathcal{J} : j > i} a_j
          b_j^*\big \rangle} {\langle b_i, Hb_i^*\rangle} b_i^* \bigg\|_{S^*} & \leq 2 \sum_{i\in
          \mathcal{J}} \eta_i
        \Big\| \sum_{j\in \mathcal{J} : j > i} a_j b_j^*\Big\|_{S^*}\\
        & \leq 2 K_u \|y\|_{S^*} \sum_{i\in \mathbb{N}} \eta_i.
      \end{aligned}
    \end{equation}
    Recall that in~\eqref{eq:proof:thm:subsymmetric:eta} we put $\eta_i = K_u^{-1}4^{-i-1}$,
    $i\in \mathbb{N}$.  Thus, inserting the estimates~\eqref{eq:U:invert:estimate:1}
    and~\eqref{eq:U:invert:estimate:2} into~\eqref{eq:crucial-identity} yields
    \begin{equation}\label{eq:U:invert:estimate:3}
      \|PHy - y\|_{S^*} \leq \frac{1}{3} \|y\|_{S^*},
      \qquad y\in Y.
    \end{equation}

    Let $J : Y\to S^*$ denote the operator given by $Jy = y$, then,
    by~\eqref{eq:U:invert:estimate:3}, $PHJ : Y\to Y$ is invertible and
    $\|(PHJ)^{-1}\|\leq \frac{3}{2}$.  Thus, if we define $V = (PHJ)^{-1}P$, the diagram
    \begin{equation}\label{eq:U:commutative-diagram}
      \vcxymatrix{
        Y \ar[rr]^{\Id_Y} \ar[rd]_{PHJ} \ar@/_{5pt}/[dd]_{J} & & Y\\
        & Y \ar[ru]^{(PHJ)^{-1}} &\\
        S^* \ar[rr]_H & & S^* \ar[lu]_{P} \ar[uu]_{V}
      }
      \qquad \|J\| \|V\| \leq 12 K_u^4 K_s^2
    \end{equation}
    is commutative.  Merging the diagrams~\eqref{eq:Y:commutative-diagram}
    and~\eqref{eq:U:commutative-diagram} concludes the proof:
    \begin{equation}\label{eq:commutative-diagram:merged}
      \vcxymatrix{%
        S^* \ar@/_{20pt}/[ddd]_{M} \ar[rrrr]^{\Id_{S^*}} \ar[d]^{B} &&&& S^*\\
        Y \ar[rrrr]^{\Id_Y} \ar@/_{5pt}/[dd]_J \ar[rrd]_{PHJ} &&&& Y \ar[u]^{Q_{|Y}}\\
        && Y \ar[rru]^{(PHJ)^{-1}} &&\\
        S^* \ar[rrrr]_H &&&& S^* \ar[llu]_{P}
        \ar[uu]_{V} \ar@/_{20pt}/[uuu]_{N}
      }
      \qquad \|M\| \|N\| \leq 48 K_u^7 K_s^4.
      \qedhere
    \end{equation}
  \end{proofstep}
\end{myproof}

\section{Direct sums of Banach spaces with a weak$^*$~subsymmetric
  basis}\label{sec:sums-w-subsymmetric}

\noindent
We begin this section by introducing notation specific to the two parameter case.  Then we extend
the subspace annihilation lemmata in Section~\ref{sec:annihilate} to the two parameter case and use them to
prove our second main result Theorem~\ref{thm:results:primary}.

\subsection{Notation}
\label{sec:notation}

Let $X$ denote a Banach space with normalized basis $(x_j)_{j=1}^\infty$, such that their normalized
associated coordinate functionals $(x_j^*)_{j=1}^\infty$ form a weak$^*$~Schauder basis of $X^*$.
For each $i\in \mathbb{N}$, let $q_i : \ell^p(X^*)\to X^*$ denote the canonical norm $1$ coordinate
projection given by
\begin{equation*}
  y_i\bigl( (y_n)_{n=1}^\infty \bigr) = y_i,
  \qquad (y_n)_{n=1}^\infty\in\ell^p(X^*).
\end{equation*}
Let $\mathcal{I}\subset \mathbb{N}$ and define $Q_{\mathcal{I}} : \ell^p(X^*)\to \ell^p(X^*)$ as the
natural norm $1$ projection onto the coordinates indexed by $\mathcal{I}$, i.e.
\begin{equation*}
  q_i Q_{\mathcal{I}} (y_n)_{n=1}^\infty = y_i,\ i\in \mathcal{I}
  \qquad\text{and}\qquad
  q_i Q_{\mathcal{I}} (y_n)_{n=1}^\infty = 0,\ i\notin \mathcal{I},
\end{equation*}
for all $(y_n)_{n=1}^\infty\in\ell^p(X^*)$.  For $i\in \mathbb{N}$, we will sometimes write $Q_i$
instead of $Q_{\{i\}}$.  For each $i\in \mathbb{N}$, we define $I_i : X^*\to \ell^p(X^*)$ as the
canonical isometric embedding of $X^*$ into the $i$-th coordinate of $\ell^p(X^*)$, which is given
by
\begin{align*}
  q_i I_i x^*
  &= x^*
    \qquad\text{and}\qquad
    q_k I_i x^*
    = 0,
    \quad k\in\mathbb{N},\ k\neq i,
  &&i\in\mathbb{N}.
\end{align*}
We define the sequence $(e_{ij})_{i,j=1}^\infty\subset \ell^p(X^*)$ by
\begin{equation*}
  e_{ij} = I_i x_j^*,
  \qquad i,j\in \mathbb{N};
\end{equation*}
hence, the associated coordinate functionals $f_{ij} : \ell^p(X^*)\to \mathbb{R}$ are given by
\begin{equation*}
  \langle f_{ij}, y\rangle
  = \langle q_i y, x_j \rangle,
  \qquad y\in \ell^p(X^*),\ i,j\in\mathbb{N}.
\end{equation*}
Given any sequence of scalars $(a_{ij})_{i,j=1}^\infty$ for which
$(\wslim_{n\to \infty} \sum_{j=1}^n a_{ij} x_j^*)_{i=1}^\infty\in \ell^p(X^*)$, we define
\begin{equation}\label{eq:sums-topology}
  \sum_{i,j=1}^\infty a_{ij} e_{ij}
  = \Bigl(\sum_{j=1}^\infty a_{ij} x_j^*\Bigr)_{i=1}^\infty
  = \Bigl(\wslim_{n\to \infty} \sum_{j=1}^n a_{ij} x_j^*\Bigr)_{i=1}^\infty.
\end{equation}
Thus, we have the identity
\begin{equation*}
  \sum_{i,j=1}^\infty \langle f_{ij}, y \rangle e_{ij}
  = \Bigl( \sum_{j=1}^\infty \langle q_i y, x_j\rangle x_j^* \Bigr)_{i=1}^\infty
  = y,
  \qquad y\in \ell^p(X^*).
\end{equation*}

For $\mathcal{K}\subset \mathbb{N}^2$, we define the projection
$R_{\mathcal{K}} : \ell^p(X^*)\to \ell^p(X^*)$ by
\begin{equation*}
  R_{\mathcal{K}} y
  = \sum_{i,j=1}^\infty \langle f_{ij}, y\rangle \charfun_{\mathcal{K}}(i,j) e_{ij},
\end{equation*}
where $\charfun_{\mathcal{K}}(i,j) = 1$ if $(i,j)\in \mathcal{K}$, and
$\charfun_{\mathcal{K}}(i,j) = 0$ if $(i,j)\notin \mathcal{K}$.  If $(x_j^*)$ is unconditional, then
$R_{\mathcal{K}}$ is bounded.

\subsection{Two subspace annihilation lemmata for direct sums}
\label{sec:subsp-annih-lemm}
The following lemmata are two parameter versions of the subspace annihilation lemmata presented in
Section~\ref{sec:annihilate}.

\begin{lem}\label{lem:past:2d}
  Let $X$ denote a Banach space with a normalized basis $(x_j)_{j=1}^\infty$, such that their
  normalized associated coordinate functionals $(x_j^*)_{j=1}^\infty$ form a weak$^*$~Schauder basis
  of $X^*$.  Let $1 \leq p \leq \infty$ and let $\Lambda_0\subset\mathbb{N}$ denote an infinite set.
  Then, given $m\in \mathbb{N}$, $y_1,\ldots,y_m\in \ell^p(X^*)$, $n_0\in \mathbb{N}$, $\eta > 0$
  and $M\in \mathbb{N}$, there exists a finite set $\mathcal{F}\subset \Lambda_0$ with
  $|\mathcal{F}| = 2M$ such that
  \begin{equation}\label{eq:lem:past:2d}
    \max_{1\leq \ell \leq m} \bigl|
    \bigl\langle \sum_{k\in \mathcal{F}} \varepsilon_{n_0k} f_{n_0k},  y_\ell \bigr\rangle
    \bigr|
    \leq \eta,
    \qquad \varepsilon \in \mathcal{E},
  \end{equation}
  where the set of signs $\mathcal{E} = \mathcal{E}(n_0,\mathcal{F})$ is given by
  \begin{equation*}
    \mathcal{E}
    = \Bigl\{ (\varepsilon_{n_0k})_{k\in \mathcal{F}} : \varepsilon_{n_0k}\in \{\pm 1\},\
    \sum_{k\in\mathcal{F}} \varepsilon_{n_0k} = 0 \Bigr\}.
  \end{equation*}
\end{lem}

\begin{proof}
  Since the proof of Lemma~\ref{lem:past:2d} is completely analogous to that of Lemma~\ref{lem:past}, we will
  omit it.
\end{proof}

\begin{lem}\label{lem:future:2d-new}
  Let $U$ denote a Banach space with a normalized unconditional basis $(u_j)_{j=1}^\infty$, such
  that their normalized associated coordinate functionals $(u_j^*)_{j=1}^\infty$ form an
  unconditional weak$^*$~Schauder basis of $U^*$.  Let $U^*$ satisfy
  condition~\eqref{eq:condition-c}, $1 \leq p \leq \infty$, and let $\Lambda_i\subset\mathbb{N}$,
  $i\in\mathbb{N}$ denote infinite sets.  Let $\varphi\in (\ell^p(U^*))^*$ and $\eta > 0$.  Then
  there exist infinite sets and $\mathcal{J}_i\subset\Lambda_i$, $i\in\mathbb{N}$ such that
  \begin{equation}\label{eq:lem:future:2d-new}
    \sup_{\|y\|_{\ell^p(U^*)}\leq 1} |\langle \varphi, R_{\{i\}\times \mathcal{J}_i} y\rangle|
    \leq \eta,
    \qquad i\in\mathbb{N}.
  \end{equation}
\end{lem}

\begin{proof}
  If we assume Lemma~\ref{lem:future:2d-new} is false, then there exists an $i_0\in\mathbb{N}$ such that
  \begin{equation}\label{eq:proof:lem:future:2d-new:1}
    \sup_{\|y\|_{\ell^p(U^*)}\leq 1} |\langle \varphi, R_{\{i_0\}\times \mathcal{A}} y\rangle|
    > \eta,
  \end{equation}
  for every infinite set $\mathcal{A}\subset\Lambda_{i_0}$.  We define
  $\theta = \frac{\eta}{2\|\varphi\|_{(\ell^p(U^*))^*}K_u}$ and choose a sequence
  $(\mathcal{A}_j)_{j=1}^\infty$ of pairwise disjoint, infinite subsets of $\mathcal{J}$ according
  to~\eqref{eq:condition-c}.  By~\eqref{eq:proof:lem:future:2d-new:1}, we can find a sequence
  $(y_j)_{j=1}^\infty$ in $\ell^p(U^*)$ with $\|y_j\| \leq 1$, $j\in\mathbb{N}$ such that
  \begin{equation}\label{eq:proof:lem:future:2d-new:2}
    \langle \varphi, R_{\{i_0\}\times \mathcal{A}_j} y_j\rangle
    > \eta,
    \qquad j\in\mathbb{N}.
  \end{equation}
  Next, define $u_j^* = q_{i_0} y_j$, $j\in\mathbb{N}$ and choose $(a_j)_{j=1}^\infty\in\ell^1$ with
  $\|(a_j)_{j=1}^\infty\|_{\ell^1} = 1$ according to~\eqref{eq:condition-c} such that
  \begin{equation}\label{eq:proof:lem:future:2d-new:3}
    \Big\| \sum_{j=1}^\infty a_j P_{\mathcal{A}_j} u_j^* \Big\|_{U^*}
    \leq \theta
    = \frac{\eta}{2\|\varphi\|_{(\ell^p(U^*))^*}K_u}.
  \end{equation}
  By~\eqref{eq:proof:lem:future:2d-new:2}, we obtain an integer $n\in\mathbb{N}$ such that
  \begin{equation}\label{eq:proof:lem:future:2d-new:4}
    \eta
    < \langle \varphi, \sum_{j=1}^n a_j R_{\{i_0\}\times \mathcal{A}_j} y_j\rangle
    \leq \|\varphi\|_{(\ell^p(U^*))^*} \Bigl\| \sum_{j=1}^n a_j R_{\{i_0\}\times \mathcal{A}_j} y_j \Bigr\|_{\ell^p(U^*)}.
  \end{equation}
  Observe that
  $\sum_{j=1}^n a_j R_{\{i_0\}\times \mathcal{A}_j} y_j = I_{i_0}\sum_{j=1}^n a_j
  P_{\mathcal{A}_j} u_j$, hence, combining~\eqref{eq:proof:lem:future:2d-new:3}
  and~\eqref{eq:proof:lem:future:2d-new:4} leads to the contradiction
  \begin{equation}\label{eq:proof:lem:future:2d-new:5}
    \eta
    \leq \|\varphi\|_{(\ell^p(U^*))^*}
    \Bigl\| \sum_{j=1}^n a_j P_{\mathcal{A}_j} u_j \Bigr\|_{U^*}
    \leq \eta/2.
    \qedhere
  \end{equation}
\end{proof}

\subsection{Proof of main result Theorem~\ref{thm:results:primary}}
\label{sec:proof-main-result}
In this section we prove our second main result Theorem~\ref{thm:results:primary}, which we repeat
here below (see Theorem~\ref{thm:primary}).  The proof involves inductively constructing a block
basis in each coordinate of $\ell^p(S^*)$, $1 \leq p \leq \infty$.  The order by which we proceed is
determined by the linear order $\prec$ on $\mathbb{N}^2$, which is defined as follows: Let $<_\ell$
denote the lexicographic order on $\mathbb{N}^2$ and define
\begin{equation*}
  (i_0,j_0) \prec (i_1,j_1)
  \quad\text{if and only if}\quad
  (i_0+j_0,i_0) <_\ell (i_1+j_1,i_1),
\end{equation*}
for all $(i_0,j_0), (i_1,j_1)\in\mathbb{N}^2$ (see Figure~\ref{fig:order}).  Let
$\mathcal{O}_\prec : \mathbb{N}^2\to \mathbb{N}$ denote the unique bijective function that preserves
the order $\prec$, i.e.
\begin{equation*}
  \mathcal{O}_\prec((i_0,j_0)) < \mathcal{O}_\prec((i_1,j_1))
  \quad\text{if and only if}\quad
  (i_0,j_0) \prec (i_1,j_1),
\end{equation*}
for all $(i_0,j_0),(i_1,j_1)\in \mathbb{N}^2$.

\begin{figure}[bth]
  \begin{center}
    \includegraphics{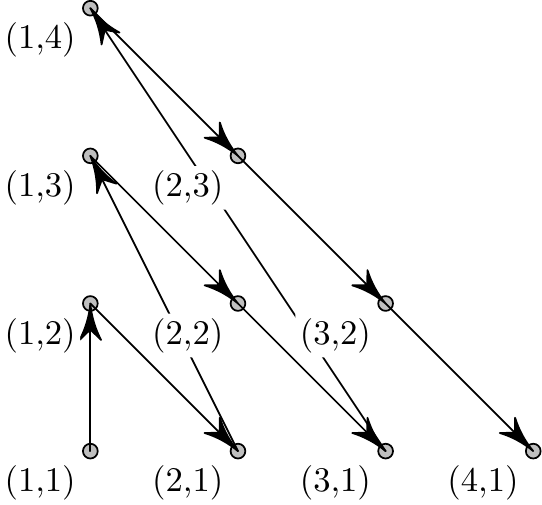}
  \end{center}
  \caption{The first 10 Elements of $\mathbb{N}^2$ with respect to the order
    $\prec$.}\label{fig:order}
\end{figure}

\begin{thm}[Main result Theorem~\ref{thm:results:primary}]\label{thm:primary}
  Let $S^*$ denote a Banach space with a subsymmetric weak$^*$~Schauder basis satisfying
  condition~\eqref{eq:condition-c}.  Then for all $1 \leq p \leq \infty$ and any given bounded
  operator $T : \ell^p(S^*)\to \ell^p(S^*)$, there exist operators
  $M,N : \ell^p(S^*)\to \ell^p(S^*)$ such that for $H=T$ or $H = \Id_{\ell^p(S^*)}-T$, the diagram
  \begin{equation}\label{eq:thm:primary}
    \vcxymatrix{\ell^p(S^*) \ar[r]^{\Id_{\ell^p(S^*)}} \ar[d]_{M} & \ell^p(S^*)\\
      \ell^p(S^*) \ar[r]_H & \ell^p(S^*) \ar[u]_{N}
    }
    \qquad \|M\| \|N\|\leq 48 K_u^7 K_s^4,
  \end{equation}
  is commutative.  Consequently, $\ell^p(S^*)$ is primary.
\end{thm}

\begin{myproof}
  For the definition of $e_{ij}$ and $f_{ij}$ we refer to Section~\ref{sec:notation}.  Within this
  proof, we identify
  \begin{equation*}
    \eta_k \leftrightarrow \eta_{k_0k_1},
    \qquad
    \mathcal{B}_k \leftrightarrow \mathcal{B}_{k_0k_1},
    \qquad
    b_k \leftrightarrow b_{k_0k_1},
    \qquad
    b_k^* \leftrightarrow b_{k_0k_1}^*,
  \end{equation*}
  whenever $\mathcal{O}_\prec((k_0,k_1)) = k$.  We conclude this preliminary step by defining
  \begin{equation}\label{eq:proof:thm:primary:eta}
    \eta_k = K_u^{-1}4^{-k-1},
    \qquad k\in \mathbb{N}.
  \end{equation}

  \begin{proofstep}[Step~\theproofstep: Inductive construction]
    We will now inductively (with respect to the order $\prec$ on $\mathbb{N}^2$) construct
    sequences $(b_{ij}^{(\varepsilon)})_{i,j=1}^\infty$ and
    $(b_{ij}^{*(\varepsilon)})_{i,j=1}^\infty$.  We begin by putting $\mathcal{J}_{1,i}=\mathbb{N}$,
    $i\in\mathbb{N}$, $\mathcal{B}_1=\{(1,1),(1,2)\}$ and define
    \begin{equation}\label{eq:proof:thm:primary:induction:begin:1}
      b_{11}^{(\varepsilon)} = e_{11} - e_{12}
      \qquad\text{and}\qquad
      b_{11}^{*(\varepsilon)} = f_{11} - f_{12}.
    \end{equation}
    By Lemma~\ref{lem:future:2d-new}, there exist infinite sets
    $\mathcal{J}_{2,i}\subset\mathbb{N}$, $i\in\mathbb{N}$ with
    \begin{equation}\label{eq:proof:thm:primary:induction:begin:2}
      \mathcal{J}_{2,i}\subset \mathcal{J}_{1,i}\setminus \{1,2\},
    \end{equation}
    such that
    \begin{equation}\label{eq:proof:thm:primary:induction:begin:3}
      \sup_{\|y\|_{\ell^p(S^*)}\leq 1}
      |\langle T^* b_1^{*(\varepsilon)}, R_{\{i\}\times\mathcal{J}_{2,i}} y\rangle|
      \leq \eta_1,
      \qquad i\in\mathbb{N}.
    \end{equation}

    Let us assume that we have already
    \begin{subequations}\label{eq:proof:thm:primary:induction:hypothesis}
      \begin{itemize}
      \item selected infinite sets $\mathcal{J}_{k,i}\subset \mathbb{N}$, $i\in\mathbb{N}$,
        $1\leq k \leq K$ and finite sets $\mathcal{B}_k\subset\mathbb{N}^2$ with
        $|\mathcal{B}_k|=2$, $1\leq k \leq K-1$ such that
        \begin{equation}\label{eq:proof:thm:primary:induction:hypothesis:a}
          \mathcal{B}_k
          = \mathcal{B}_{k_0k_1}
          \subset (\{k_0\}\times\mathcal{J}_{k,k_0}) \setminus(\{k_0\}\times\mathcal{J}_{k+1,k_0}),
        \end{equation}
        whenever $1 \leq k \leq K-1$ and $\mathcal{O}_\prec((k_0,k_1)) = k$, and that
        \begin{equation}\label{eq:proof:thm:primary:induction:hypothesis:b}
          \mathcal{J}_{k+1,i}
          \subset \mathcal{J}_{k,i}\mathbin{\Big\backslash}
          \Bigl\{ j\in \mathbb{N} :
          j \leq \max\bigl\{ j'\in \mathbb{N} :
          \exists i'\in\mathbb{N},\ (i',j')\in\bigcup_{\ell=1}^k\mathcal{B}_\ell
          \bigr\}
          \Bigr\},
        \end{equation}
        for all $i\in\mathbb{N}$ and $1 \leq k \leq K-1$;

      \item defined $b_k^{(\varepsilon)}$, $b_k^{*(\varepsilon)}$ by
        \begin{equation}\label{eq:proof:thm:primary:induction:hypothesis:c}
          b_k^{(\varepsilon)}
          = e_{k_0 j} - e_{k_0 j'}
          \qquad\text{and}\qquad
          b_k^{*(\varepsilon)}
          = f_{k_0 j} - f_{k_0 j'},
        \end{equation}
        whenever $1\leq k \leq K-1$, $\mathcal{O}_\prec((k_0,k_1)) = k$ and
        $(k_0,j), (k_0,j')\in\mathcal{B}_k=\mathcal{B}_{k_0k_1}$ with $j<j'$;

      \item obtained the estimates
        \begin{align}
          \sum_{\ell=1}^{k-1} |\langle b_k^{*(\varepsilon)}, T b_\ell^{(\varepsilon)}\rangle|
          &\leq \eta_k,
            \label{eq:proof:thm:primary:induction:hypothesis:d}\\
          \sup_{\|y\|_{\ell^p(S^*)}\leq 1}
          \bigl| \bigl\langle
          T^* b_k^{*(\varepsilon)}, R_{\{i\}\times\mathcal{J}_{k+1,i}} y
          \bigr\rangle\bigr|
          &\leq \eta_k,
          &&i\in\mathbb{N},
             \label{eq:proof:thm:primary:induction:hypothesis:e}
        \end{align}
        for all $1\leq k \leq K-1$.
      \end{itemize}
    \end{subequations}
    We now come to the inductive step.  To this end, choose $(K_0,K_1)\in\mathbb{N}^2$ with
    $\mathcal{O}_\prec((K_0,K_1)) = K$.  The set $\mathcal{J}_{K,K_0}$ is infinite
    (see~\eqref{eq:proof:thm:primary:induction:hypothesis:a}), hence, by Lemma~\ref{lem:past:2d}
    there exists a set $\mathcal{F}\subset \mathcal{J}_{K,K_0}$ with $|\mathcal{F}| = 2$ such that
    \begin{equation}\label{eq:proof:thm:primary:induction:step:1}
      \sum_{\ell=1}^{K-1} \bigl| \bigl\langle
      f_{K_0 j} - f_{K_0 j'},
      Tb_\ell^{(\varepsilon)}
      \bigr\rangle \bigr|
      \leq \eta_K,
      \qquad j,j'\in \mathcal{F}.
    \end{equation}
    We define the collection
    \begin{equation}\label{eq:proof:thm:primary:induction:step:2}
      \mathcal{B}_K
      = \mathcal{B}_{K_0K_1}
      = \{K_0\}\times \mathcal{F},
    \end{equation}
    and note that (with the possible exception of
    $\mathcal{B}_K\cap (\{K_0\}\times\mathcal{J}_{K+1,K_0}) = \emptyset$)
    condition~\eqref{eq:proof:thm:primary:induction:hypothesis:a} holds true for $K+1$.  We will
    take care of the possible exception by appropriately choosing $\mathcal{J}_{K+1,i}$,
    $i\in\mathbb{N}$ at a later stage of the proof.  Next, we define the functions
    $b_K^{(\varepsilon)}$, $b_K^{*(\varepsilon)}$ and the signs $\varepsilon_{K_0,j}$,
    $j\in \mathcal{F}$ by
    \begin{equation}\label{eq:proof:thm:primary:induction:step:3}
      b_K^{(\varepsilon)}
      = e_{K_0 j} - e_{K_0 j'},
      \qquad
      b_K^{*(\varepsilon)}
      = f_{K_0 j} - f_{K_0 j'},
      \qquad j,j'\in \mathcal{F}, j < j',
    \end{equation}
    which is in accordance with~\eqref{eq:proof:thm:primary:induction:hypothesis:c}.
    Combining~\eqref{eq:proof:thm:primary:induction:step:1}
    with~\eqref{eq:proof:thm:primary:induction:step:3} yields
    \begin{equation}\label{eq:proof:thm:primary:induction:step:3}
      \sum_{\ell=1}^{K-1}
      \bigl| \bigl\langle b_K^{*(\varepsilon)}, Tb_\ell^{(\varepsilon)} \bigr\rangle \bigr|
      \leq \eta_K,
    \end{equation}
    showing that~\eqref{eq:proof:thm:primary:induction:hypothesis:d} holds for $K+1$, as required.
    Since all the sets $\mathcal{B}_k$, $1\leq k\leq K$ are finite, we can apply
    Lemma~\ref{lem:future:2d-new} to find infinite sets $\mathcal{J}_{K+1,i}$, $i\in\mathbb{N}$ with
    \begin{equation}\label{eq:proof:thm:primary:induction:step:4}
      \mathcal{J}_{K+1,i}
      \subset \mathcal{J}_{K,i}\mathbin{\Big\backslash}
      \Bigl\{ j\in \mathbb{N} :
      j \leq \max\bigl\{ j'\in \mathbb{N} :
      \exists i'\in\mathbb{N},\ (i',j')\in\bigcup_{\ell=1}^K\mathcal{B}_\ell
      \bigr\}
      \Bigr\},
    \end{equation}
    such that the following estimate is satisfied:
    \begin{equation}\label{eq:proof:thm:primary:induction:step:5}
      \sup_{\|y\|_{\ell^p(S^*)}\leq 1}
      \bigl| \bigl\langle
      T^* b_K^{*(\varepsilon)}, R_{\{i\}\times\mathcal{J}_{K+1,i}} y
      \bigr\rangle\bigr|
      \leq \eta_K,
      \qquad i\in\mathbb{N},
    \end{equation}
    Note that~\eqref{eq:proof:thm:primary:induction:step:4}
    and~\eqref{eq:proof:thm:primary:induction:step:5}
    imply~\eqref{eq:proof:thm:primary:induction:hypothesis:b}
    and~\eqref{eq:proof:thm:primary:induction:hypothesis:e} for $K+1$.
    Moreover,~\eqref{eq:proof:thm:primary:induction:step:4} also implies
    $\mathcal{B}_K\cap (\{K_0\}\times\mathcal{J}_{K+1,K_0}) = \emptyset$, which completes the
    missing part of~\eqref{eq:proof:thm:primary:induction:hypothesis:a} for $K+1$
    (see~\eqref{eq:proof:thm:primary:induction:step:2}).
    
    Altogether, we completed the inductive step, and we know
    that~\eqref{eq:proof:thm:primary:induction:hypothesis} is true for $K+1$.
  \end{proofstep}

  \begin{proofstep}[Step~\theproofstep: Basic operators]
    We define the operators $B,Q : \ell^p(S^*)\to \ell^p(S^*)$ by
    \begin{subequations}\label{eq:proof:thm:primary:BQ}
      \begin{align}
        B y
        & = \sum_{i,j=1}^\infty \langle f_{ij}, y \rangle \frac{1}{\|b_{ij}^{(\varepsilon)}\|_{\ell^p(S^*)}} b_{ij}^{(\varepsilon)},
          \qquad y\in \ell^p(S^*),
          \label{eq:proof:thm:primary:BQ:a}\\
        Q y
        & = \sum_{i,j=1}^\infty
          \frac{\|b_{ij}^{(\varepsilon)}\|_{\ell^p(S^*)}}{|\mathcal{B}_{ij}|}
          \langle b_{ij}^{*(\varepsilon)}, y \rangle e_{ij},
          \qquad y\in \ell^p(S^*).
          \label{eq:proof:thm:primary:BQ:b}
      \end{align}
    \end{subequations}
    For the mode of convergence of the above series, we refer to~\eqref{eq:sums-topology}.

    We will now show that $B$ and $Q$ are bounded linear operators.  To this end, let
    $y\in \ell^p(S^*)$ and observe that by~\eqref{eq:proof:thm:primary:induction:hypothesis:c},
    $\|By\|_{\ell^p(S^*)}$ is dominated by
    \begin{equation*}
      \Big\|
      \sum_{i,j=1}^\infty \langle f_{ij}, y \rangle \frac{1}{\|b_{ij}^{(\varepsilon)}\|_{\ell^p(S^*)}} e_{k_{ij} \ell_{ij}}
      \Big\|_{\ell^p(S^*)}
      + \Big\|
      \sum_{i,j=1}^\infty \langle f_{ij}, y \rangle \frac{1}{\|b_{ij}^{(\varepsilon)}\|_{\ell^p(S^*)}} e_{k_{ij} m_{ij}}
      \Big\| _{\ell^p(S^*)},
    \end{equation*}
    where $\{(k_{ij}, \ell_{ij}),(k_{ij}, m_{ij})\} = \mathcal{B}_{ij}$, $i,j\in \mathbb{N}$.  By
    unconditionality we obtain that $\|b_{ij}^{(\varepsilon)}\|_{\ell^p(S^*)} \geq K_u^{-1}$,
    $i,j\in \mathbb{N}$, and subsequently
    \begin{equation*}
      \|By\|_{\ell^p(S^*)}
      \leq K_u^2 \Big\|
      \sum_{i,j=1}^\infty \langle f_{ij}, y \rangle  e_{k_{ij} \ell_{ij}}
      \Big\|_{\ell^p(S^*)}
      + K_u^2 \Big\|
      \sum_{i,j=1}^\infty \langle f_{ij}, y \rangle  e_{k_{ij} m_{ij}}
      \Big\| _{\ell^p(S^*)}.
    \end{equation*}
    Furthermore, by~\eqref{eq:proof:thm:primary:induction:hypothesis:b} we have that
    $\ell_{ij} < \ell_{i(j+1)}$ and $m_{ij} < m_{i(j+1)}$, $i,j\in \mathbb{N}$.  The
    weak$^*$~Schauder basis $(s_j^*)_{j=1}^\infty$ of $S^*$ is subsymmetric, hence
    \begin{equation}\label{eq:proof:thm:primary:B:estimate}
      \begin{aligned}
        \|By\|_{\ell^p(S^*)} &\leq 2 K_u^2K_s \Big\| \sum_{i,j=1}^\infty \langle f_{ij}, y \rangle
        e_{k_{ij} j} \Big\|_{\ell^p(S^*)} = 2 K_u^2K_s \Big\| \sum_{i,j=1}^\infty \langle f_{ij}, y
        \rangle e_{ij}
        \Big\|_{\ell^p(S^*)}\\
        &= 2 K_u^2K_s\| y \|_{\ell^p(S^*)}.
      \end{aligned}
    \end{equation}

    Similarly, we obtain that $\|Qy\|_{\ell^p(S^*)}$ is dominated by
    \begin{equation*}
      \Big\|
      \sum_{i,j=1}^\infty \frac{\|b_{ij}^{(\varepsilon)}\|_{\ell^p(S^*)}}{|\mathcal{B}_{ij}|}
      \langle f_{k_{ij}\ell_{ij}}, y \rangle e_{ij}
      \Big\|_{\ell^p(S^*)}
      + \Big\|
      \sum_{i,j=1}^\infty \frac{\|b_{ij}^{(\varepsilon)}\|_{\ell^p(S^*)}}{|\mathcal{B}_{ij}|}
      \langle f_{k_{ij}m_{ij}}, y \rangle e_{ij}
      \Big\| _{\ell^p(S^*)},
    \end{equation*}
    which, since $(s_j^*)_{j=1}^\infty$ is subsymmetric, gives us the following upper bound for
    $\|Qy\|_{\ell^p(S^*)}$:
    \begin{equation*}
      K_s \Big\|
      \sum_{i,j=1}^\infty \frac{\|b_{ij}^{(\varepsilon)}\|_{\ell^p(S^*)}}{|\mathcal{B}_{ij}|}
      \langle f_{k_{ij}\ell_{ij}}, y \rangle e_{i \ell_{i_j}}
      \Big\|_{\ell^p(S^*)}
      + K_s \Big\|
      \sum_{i,j=1}^\infty \frac{\|b_{ij}^{(\varepsilon)}\|_{\ell^p(S^*)}}{|\mathcal{B}_{ij}|}
      \langle f_{k_{ij}m_{ij}}, y \rangle e_{i m_{ij}}
      \Big\| _{\ell^p(S^*)}.
    \end{equation*}
    Using $\frac{\|b_{ij}^{(\varepsilon)}\|_{\ell^p(S^*)}}{|\mathcal{B}_{ij}|}\leq 1$, we obtain by
    unconditionality
    \begin{equation}\label{eq:proof:thm:primary:Q:estimate}
      \|Qy\|_{\ell^p(S^*)}
      \leq 2 K_s K_u \|y\|_{\ell^p(S^*)}.
    \end{equation}

    One can easily verify that $Q B = \Id_{\ell^p(S^*)}$, i.e.~the diagram
    \begin{equation}\label{eq:proof:thm:primary:BQ:commutative-diagram}
      \vcxymatrix{\ell^p(S^*) \ar[rr]^{\Id_{\ell^p(S^*)}} \ar[dr]_{B} & & \ell^p(S^*)\\
        & \ell^p(S^*) \ar[ru]_{Q} & }
    \end{equation}
    is commutative.  Consequently, $B$ is an isomorphism onto its range, and its range is
    complemented by $BQ$.
  \end{proofstep}

  \begin{proofstep}[Step~\theproofstep: Factorization of the identity]
    With $i\in\mathbb{N}$ fixed, observe that at least one of the two following sets is infinite:
    \begin{equation*}
      \mathcal{K}_i = \{j\in \mathbb{N} : |\langle b_{ij}^{*(\varepsilon)}, Tb_{ij}^{(\varepsilon)}\rangle| \geq 1\},
      \qquad
      \mathcal{L}_i = \{j\in \mathbb{N} :
      |\langle b_{ij}^{*(\varepsilon)}, (\Id_{\ell^p(S^*)}-T) b_{ij}^{(\varepsilon)}\rangle| \geq 1\}.
    \end{equation*}
    Thus, clearly one of the following two sets is infinite:
    \begin{equation*}
      \{ i\in\mathbb{N} : \mathcal{K}_i\ \text{is infinite}\}
      \qquad\text{or}\qquad
      \{ i\in\mathbb{N} : \mathcal{L}_i\ \text{is infinite}\}
    \end{equation*}
    If the left set is infinite, we denote it by $\mathcal{I}$ and we put
    $\mathcal{J}_i = \mathcal{K}_i$, $i\in\mathcal{I}$ as well as $H=T$; if the left set is finite,
    we denote the right set by $\mathcal{I}$, and we define $\mathcal{J}_i = \mathcal{L}_i$,
    $i\in\mathcal{I}$ as well as $H=\Id_{\ell^p(S^*)}-T$.  In either of these two cases, we obtain
    that
    \begin{equation}\label{eq:proof:thm:primary:either-or-large}
      \mathcal{I}\ \text{is infinite},
      \quad \mathcal{J}_i\ \text{is infinite},\ i\in\mathcal{I}
      \quad\text{and}\quad
      |\langle b_{ij}^{*(\varepsilon)}, Hb_{ij}^{(\varepsilon)}\rangle| \geq 1,
      \ i\in\mathcal{I}, j\in\mathcal{J}_i.
    \end{equation}
    Let $Y = B(\ell^p(S^*))$ and define $A : \ell^p(S^*)\to Y$, $y\mapsto By$.  Thus,
    by~\eqref{eq:proof:thm:primary:B:estimate},~\eqref{eq:proof:thm:primary:Q:estimate}
    and~\eqref{eq:proof:thm:primary:BQ:commutative-diagram}, the following diagram is commutative:
    \begin{equation}\label{eq:proof:thm:primary:Y:commutative-diagram}
      \vcxymatrix{
        \ell^p(S^*) \ar[rr]^{\Id_{\ell^p(S^*)}} \ar[d]_A & & \ell^p(S^*)\\
        Y \ar[rr]_{\Id_{\ell^p(S^*)}} & & Y \ar[u]_{Q_{|Y}}
      }
      \qquad \|A\| \|Q_{|Y}\| \leq 4 K_u^3 K_s^2.
    \end{equation}

    Now, put $\mathcal{K} = \{\mathcal{O}_\prec(i,j) : i\in\mathcal{I},\ j\in\mathcal{J}_i\}$ and
    define $P : \ell^p(S^*)\to Y$ by
    \begin{equation}\label{eq:proof:thm:primary:U}
      Py
      = \sum_{k\in\mathcal{K}}
      \frac{\langle b_k^{*(\varepsilon)}, y\rangle}
      {\langle b_k^{*(\varepsilon)}, H b_k^{(\varepsilon)}\rangle} b_k^{(\varepsilon)},
      \qquad y\in \ell^p(S^*).
    \end{equation}
    Recall that at the beginning of the proof we identified $b_k^{(\varepsilon)}$ with
    $b_{k_0k_1}^{(\varepsilon)}$ and $b_k^{*(\varepsilon)}$ with $b_{k_0k_1}^{*(\varepsilon)}$,
    whenever $\mathcal{O}_\prec((k_0,k_1)) = k$.  Observe that
    by~\eqref{eq:proof:thm:primary:either-or-large}, unconditionality and the definition of $B$ and
    $Q$ (see~\eqref{eq:proof:thm:primary:BQ}), we obtain that
    \begin{equation*}
      \|Py\|_{\ell^p(S^*)}
      \leq 2 K_u \|BQy\|_{\ell^p(S^*)}
    \end{equation*}
    Combining the latter estimate with~\eqref{eq:proof:thm:primary:B:estimate}
    and~\eqref{eq:proof:thm:primary:Q:estimate} yields
    \begin{equation}\label{eq:proof:thm:primary:U:estimate}
      \|Py\|_{\ell^p(S^*)}
      \leq 8 K_u^4 K_s^2 \|y\|_{\ell^p(S^*)}.
    \end{equation}

    A straightforward calculation shows that for all
    $y = \sum_{k\in\mathcal{K}} a_k b_k^{(\varepsilon)} \in Y$, the following identity is true:
    \begin{equation}\label{eq:proof:thm:primary:crucial-identity}
      PHy - y
      = \sum_{k\in \mathcal{K}}
      \sum_{\ell\in \mathcal{K} : \ell < k}
      a_\ell
      \frac{\langle b_k^{*(\varepsilon)}, H b_\ell^{(\varepsilon)}\rangle}
      {\langle b_k^{*(\varepsilon)}, Hb_k^{(\varepsilon)}\rangle}
      b_k^{(\varepsilon)}
      + \sum_{k\in \mathcal{K}}
      \frac{\big\langle H^* b_k^{*(\varepsilon)},
        \sum_{\ell\in \mathcal{K} : \ell > k} a_\ell b_\ell^{(\varepsilon)}\big\rangle}
      {\langle b_k^{*(\varepsilon)}, Hb_k^{(\varepsilon)}\rangle}
      b_k^{(\varepsilon)}.
    \end{equation}
    Since
    $|\langle b_j^{(\varepsilon)}, y\rangle| \leq \|b_j^{(\varepsilon)}\|_S \|y\|_{\ell^p(S^*)}$, we
    obtain $|a_j| \leq \|y\|_{\ell^p(S^*)}$.  Hence, by~\eqref{eq:proof:thm:primary:either-or-large}
    and the off-diagonal estimate~\eqref{eq:proof:thm:primary:induction:hypothesis:d}, we obtain
    \begin{equation}\label{eq:proof:thm:primary:U:invert:estimate:1}
      \bigg\|\sum_{k\in \mathcal{K}}
      \sum_{\ell\in \mathcal{K} : \ell < k}
      a_\ell
      \frac{\langle b_k^{*(\varepsilon)}, H b_\ell^{(\varepsilon)}\rangle}
      {\langle b_k^{*(\varepsilon)}, Hb_k^{(\varepsilon)}\rangle}
      b_k^{(\varepsilon)}
      \bigg\|_{\ell^p(S^*)}
      \leq 2\|y\|_{\ell^p(S^*)} \sum_{k\in \mathbb{N}} \eta_k.
    \end{equation}
    Using~\eqref{eq:proof:thm:primary:either-or-large},~\eqref{eq:proof:thm:primary:induction:hypothesis:b},~\eqref{eq:proof:thm:primary:induction:hypothesis:a}
    and the other off-diagonal estimate~\eqref{eq:proof:thm:primary:induction:hypothesis:e} yields
    \begin{equation}\label{eq:proof:thm:primary:U:invert:estimate:2}
      \bigg\|\sum_{k\in \mathcal{K}}
      \frac{\big\langle H^*b_k^{*(\varepsilon)}, \sum_{\ell\in \mathcal{K} : \ell > k} a_\ell b_\ell^{(\varepsilon)}\big \rangle}
      {\langle b_k^{*(\varepsilon)}, Hb_k^{(\varepsilon)}\rangle}
      b_k^{(\varepsilon)} \bigg\|_{\ell^p(S^*)}
      \leq 2 \sum_{k\in \mathcal{K}} \eta_k
      \Big\| \sum_{\ell\in \mathcal{K} : \ell > k} a_\ell b_\ell^{(\varepsilon)}\Big\|_{\ell^p(S^*)}.
    \end{equation}
    By~\eqref{eq:proof:thm:primary:BQ:a}, estimate~\eqref{eq:proof:thm:primary:B:estimate} and by
    unconditionality, we obtain
    \begin{equation*}
      \Big\| \sum_{\ell\in \mathcal{K} : \ell > k} a_\ell b_\ell^{(\varepsilon)}\Big\|_{\ell^p(S^*)}
      \leq 2 K_u^3K_s \|y\|_{\ell^p(S^*)}.
    \end{equation*}
    The latter estimate together with~\eqref{eq:proof:thm:primary:U:invert:estimate:2} yields
    \begin{equation}\label{eq:proof:thm:primary:U:invert:estimate:2-1}
      \bigg\|\sum_{k\in \mathcal{K}}
      \frac{\big\langle H^*b_k^{*(\varepsilon)}, \sum_{\ell\in \mathcal{K} : \ell > k} a_\ell b_\ell^{(\varepsilon)}\big \rangle}
      {\langle b_k^{*(\varepsilon)}, Hb_k^{(\varepsilon)}\rangle}
      b_k^{(\varepsilon)} \bigg\|_{\ell^p(S^*)}
      \leq 4 K_u^3K_s \sum_{k\in \mathcal{K}} \eta_k \|y\|_{\ell^p(S^*)}.
    \end{equation}

    Recall that we put $\eta_k = K_u^{-1}4^{-k-1}$, $k\in \mathbb{N}$,
    (see~\eqref{eq:proof:thm:primary:eta}).  Combining the
    estimates~\eqref{eq:proof:thm:primary:U:invert:estimate:1}
    and~\eqref{eq:proof:thm:primary:U:invert:estimate:2-1}
    with~\eqref{eq:proof:thm:primary:crucial-identity} yields
    \begin{equation}\label{eq:proof:thm:primary:U:invert:estimate:3}
      \|PHy - y\|_{\ell^p(S^*)} \leq \frac{1}{3} \|y\|_{\ell^p(S^*)},
      \qquad y\in Y.
    \end{equation}

    Let $J : Y\to \ell^p(S^*)$ denote the operator given by $Jy = y$, then $PHJ : Y\to Y$ is
    invertible by~\eqref{eq:proof:thm:primary:U:invert:estimate:3}.  Thus, if we define
    $V = (PHJ)^{-1}P$, the following diagram is commutative:
    \begin{equation}\label{eq:proof:thm:primary:U:commutative-diagram}
      \vcxymatrix{
        Y \ar@/^/[rr]^{\Id_Y} \ar[dd]_J & & Y\\
        & Y \ar[ru]^{(PHJ)^{-1}} &\\
        \ell^p(S^*) \ar[rr]_H & & \ell^p(S^*) \ar[lu]_{P} \ar[uu]_{V}
      }
      \qquad \|J\| \|V\|\leq 12 K_u^4 K_s^2.
    \end{equation}
    The operators $J$ and $V$ are both bounded, thus merging the
    diagrams~\eqref{eq:proof:thm:primary:Y:commutative-diagram}
    and~\eqref{eq:proof:thm:primary:U:commutative-diagram} yields the commutative diagram
    \begin{equation}\label{eq:proof:thm:primary:P:commutative-diagram}
      \vcxymatrix{
        \ell^p(S^*) \ar@/_{20pt}/[ddd]_M \ar[rr]^{\Id_{\ell^p(S^*)}} \ar[d]_A & & \ell^p(S^*)\\
        Y \ar@/^/[rr]^{\Id_Y} \ar[dd]_J & & Y \ar[u]^{Q_{|Y}}\\
        & Y \ar[ru]^{(PHJ)^{-1}} &\\
        \ell^p(S^*) \ar[rr]_H & & \ell^p(S^*) \ar[lu]_P \ar[uu]_V \ar@/_{20pt}/[uuu]_N
      }
      \qquad \|M\| \|N\|\leq 48 K_u^7 K_s^4.
    \end{equation}
  \end{proofstep}

  \begin{proofstep}[Step~\theproofstep: $\ell^p(S^*)$ is primary]
    Let $Q : \ell^p(S^*)\to \ell^p(S^*)$ denote a bounded projection.  Then
    by~\eqref{eq:thm:primary}, we know that $\ell^p(S^*)$ is either isomorphic to a complemented
    subspace of $Q(\ell^p(S^*))$, or isomorphic to a complemented subspace of
    $(\Id_{\ell^p(S^*)} - Q)(\ell^p(S^*))$.  Moreover, since $\ell^p(S^*)$ is the $\ell^p$-sum of a
    Banach space, $\ell^p(S^*)$ is isomorphic to $\ell^p(\ell^p(S^*))$.  Thus, by
    Pe{\l}czy{\'n}ski's decomposition method (see~\cite{pelczynski:1960}; see
    also~\cite[II.B.24]{wojtaszczyk:1991}), we obtain that either $Q(\ell^p(S^*))$ is isomorphic to
    $\ell^p(S^*)$, or $(\Id_{\ell^p(S^*)} - Q)(\ell^p(S^*))$ is isomorphic to $\ell^p(S^*)$.\qedhere
  \end{proofstep}
\end{myproof}

\subsection*{Acknowledgments}\hfill\\
\noindent
It is my pleasure to thank P.F.X.~Müller for many helpful discussions.  Supported by the Austrian
Science Foundation (FWF) Pr.Nr. P28352.

\bibliographystyle{abbrv}
\bibliography{bibliography}
% \nocite{*}

\end{document}